\documentclass[12pt]{amsart}
\usepackage[utf8]{inputenc}
\usepackage{amsmath}
\usepackage{bbm}

\usepackage{latexsym}
\usepackage{xcolor}
\usepackage{enumerate}
\usepackage{xcolor}
\usepackage{epsfig}
\usepackage{epstopdf}

 \usepackage[usenames,dvipsnames]{pstricks}
 \usepackage{epsfig}
 \usepackage{pst-grad} 
\usepackage{pst-plot} 

\newtheorem{thm}{Theorem}
\newtheorem{lemma}[thm]{Lemma}

\newtheorem{prop}[thm]{Proposition}

\theoremstyle{remark}

\newcommand{\beq}{\begin{equation}}
\newcommand{\eeq}{\end{equation}}

\usepackage{xcolor}

\newcommand{\ii}{{\rm i}}

\newcommand{\RR}{\mathbb{R}}
\newcommand{\EE}{\mathbb{E}}

\newcommand{\CC}{\mathbb{C}}

\newcommand{\ol}{\overline}

\def\Ima{\mathrm{Im}\,}
\def\Rea{\mathrm{Re}\,}

\title[Zeros of harmonic polynomials]{Experiments on the zeros of harmonic polynomials using certified counting}

\author[J. D. Hauenstein, A. Lerario, E. Lundberg, D. Mehta]{Jonathan D. Hauenstein, Antonio Lerario, \\ Erik Lundberg, and Dhagash Mehta}
\address{Jonathan D. Hauenstein \\
         Department of Mathematics\\
         North Carolina State University\\
         Raleigh\\
         North Carolina \ 27695\\
         USA}
\email{hauenstein@ncsu.edu}
\urladdr{\tt www.math.ncsu.edu/~jdhauens}

\address{Antonio Lerario \\
Institut Camille Jordan \\
Université Claude Bernard Lyon 1 \\
43 boulevard du 11 novembre 1918 \\
69622 Villeurbanne cedex \\
France}
\email{lerario@math.univ-lyon1.fr}
\urladdr{\tt math.univ-lyon1.fr/~lerario/homepage/Home.html}

\address{Erik Lundberg
Department of Mathematics \\
Purdue University \\
150 N. University Street \\
West Lafayette \\
Indiana \ 47907 \\
USA
}
\email{elundber@math.purdue.edu}
\urladdr{\tt www.math.purdue.edu/~elundber}

\address{Dhagash Mehta \\
         Department of Mathematics\\
         North Carolina State University\\
         Raleigh\\
         North Carolina \ 27695\\
         USA}
\email{dbmehta@ncsu.edu}
\urladdr{\tt www.math.ncsu.edu/~dbmehta}
\thanks{Research of Hauenstein supported in part by NSF grant DMS-1262428 and DARPA YFA.
Research of Lerario supported by the
 European Community's Seventh Framework Programme ([FP7/2007-2013] [FP7/2007-2011]) under grant agreement No. [258204].  
        Research of Mehta supported in part by DARPA YFA.
        }

\begin{document}

\begin{abstract}

Motivated by Wilmshurst's conjecture, 
we investigate the zeros of harmonic polynomials.
We utilize a \emph{certified counting} approach
which is a combination of two methods from numerical algebraic geometry:
numerical polynomial homotopy continuation to compute a numerical 
approximation of each zero and Smale's alpha-theory to certify the results.  
Using this approach, we provide new examples of harmonic polynomials having the most extreme number of zeros known so far; we also study the mean and variance 
of the number of zeros of random harmonic polynomials.

\end{abstract}

\maketitle

\section{Introduction}
\subsection{Harmonic polynomials and Wilmshurst's conjecture}
A harmonic polynomial is a complex-valued harmonic function defined by:
\beq \label{eq:harmonic}F(z)=p(z)+\overline{q(z)},\eeq
where $p$ and $q$ are polynomials of degree respectively $n$ and $m$, with $n\geq m\geq 0$.

The zeros of $F$ are the points $z\in\CC$ such that $F(z) = 0$.  One approach for
computing such zeros is to treat $z = x + y\cdot\sqrt{-1}$ for $x,y\in\RR$
and consider the intersection of the real curves in $\RR^2$ defined by
$\textrm{Re}(F(x+y\cdot\sqrt{-1}))= 0$ and $\textrm{Im}(F(x+y\cdot\sqrt{-1}))=0$.
Since each of these curves has degree at most~$n$, B\'ezout's 
theorem yields that the number $N_F$ of isolated zeros of $F$ is bounded above by $n^2$.

A simple argument from degree theory shows that $F$ always has at least $n$ zeros, 
which is sharp for every $m$ and $n$.  However, this still does not rule out the possibility 
that $F$ vanishes on a curve, e.g., $F(z) = z^n + \overline{z}^n$.  The key to this example
having infinitely many solutions is $m = n$.  When $n > m$, Wilmshurst \cite{Wilm98}
showed that each solution is always isolated.  
By treating $z$ and $\overline{z}$ as independent variables in $\CC$, 
this finiteness can be observed from seeing that the leading monomials (with respect to total degree) 
of $F$ and $\overline{F}$ are $z^n$ and $\overline{z}^n$ generate a zero-dimensional ideal. 
He also made the conjecture that B\'ezout's bound can be refined to:
\beq \label{eq:wilm} N_F\leq 3n-2+m(m-1)\qquad \textrm{(Wilmshurst's conjecture)}\eeq
This conjecture is stated in \cite[Remark 2]{Wilm98}, discussed further in \cite{Sh2002},
and listed as an open problem in \cite{BL2010}.
For $m=n-1$, the upper bound follows from Wilmshurst's theorem \cite{Wilm98},
and examples were also given in~\cite{Wilm98} showing that this bound is sharp (shown independently in \cite{BHS1995}).
For $m=1$, the upper bound was shown by Khavinson and Swiatek \cite{K-S} using holomorphic dynamics.

In the $m=n-3$ case, the conjectured bound is $3n-2 + m(m-1) = n^2 - 4n + 10$
with \emph{counterexamples} provided in \cite{LLL} for which $N_F > n^2 - 3 n + O(1)$.
On the other hand, it was additionally suggested in \cite{LLL} that some modified version of the conjecture might be true.
Despite the ongoing interest in Wilmshurst's conjecture, 
there has been no progress on improving the B\'ezout bound for $1<m<n-1$.

One study that does relate to intermediate values of $m$ is the 
work of Li and Wei \cite{LiWei2009} where a \emph{random} approach is suggested. 
More precisely, they took $p$ and~$q$ to be a complex version of the \emph{Kostlan} model (see \cite{EK1995}), namely
\begin{equation} \label{eq:LW}
p(z) = \sum_{k=0}^{n} a_k \sqrt{\binom{n}{k}} z^k, \quad q(z) = \sum_{k=0}^{m} b_k \sqrt{\binom{m}{k}} z^k,
\end{equation}
where $a_k$ and $b_k$ are i.i.d. complex Gaussians.

In the two cases $m = n$ and $m = \alpha n + o(n)$ with $0 \leq \alpha < 1$, the choices of $p$ and $q$ 
in (\ref{eq:LW}) lead to (respectively):
\beq \EE N_F \sim \frac{\pi}{4} n^{3/2}\qquad \textrm{and}\qquad \EE N_F\sim n.\eeq
Notice that when $m=\alpha n+o(n)$ the average number of zeros is asymptotically the \emph{fewest possible}.
The approach of \cite{LiWei2009} used to obtain the asymptotic expectation of $N_F$
seems insufficient, as they point out, to study the~variance~of~$N_F$.

\subsection{Experiments}
The purpose of this note is to use certified numerical computations to shed light on the 
zeros of $F$ for both the deterministic~and~random~side.

These results rely upon the \emph{method} we used: numerical polynomial homotopy continuation
via {\tt Bertini} \cite{Bertini-package} with results certified by Smale's $\alpha$-theory \cite{BCSS} via {\tt alphaCertified} \cite{2010arXiv1011.1091H}.  For the problems
under consideration, this combination allows us to provably find all solutions for the systems under consideration.

Section \ref{s:extremal} concerns the deterministic part: we provide more counterexamples to \eqref{eq:wilm}. 
Specifically, we extend the construction from \cite{LLL} of extremal examples which were shown to violate the Wilmshurst conjecture when $m=n-3$.  Here, we show that for all $n \leq 20$, and $m\leq n-2$, 
a similar construction exceeds this conjectured bound (see Figure~\ref{fig:extreme}).
Presently, these are the most extreme examples known, and provide new insight on maximum number of zeros.
The computer-assisted proof via $\alpha$-theory is used for a finite number of cases 
with an analytic approach proving some lower estimates for infinitely many cases.

Section \ref{s:random} discusses our random study.
The data is based on $1000$ trials for each stated choice of $m$ and $n$.
In each trial, the zeros were provably determined in order to avoid 
systematic error in the data.

Based on these simulations, 
we conjecture that the variance of the number of zeros for $m = n$ in the Li-Wei model is $\Theta(n^2)$.
We pose a problem to modify $q$ in the Li-Wei definition in order to obtain more zeros on average when $m=\alpha n$,
i.e., to find a ``richer'' definition of \emph{random harmonic polynomial}.
We propose one alternative choice where $q$ is taken from a truncated version of the Kostlan ensemble,
and we conjecture that as $m \rightarrow \infty$ with $m = \alpha n$,
the mean number of zeros is $\Theta(m^{3/2})$.

\subsection{Certified counting of the zeros}\label{sec:Certified}

Theoretically, numerical polynomial homotopy continuation (see \cite{Bertini-book,SW:05} for general overview)
is an approach to compute \textit{all} isolated solutions to a given system of multivariate nonlinear polynomial equations.
Based on the current application, we simply use a bivariate B\'ezout (or total degree) homotopy with
each polynomial having degree $n$ which obtains the maximum number of isolated zeros, namely $n^2$.  
One constructs a homotopy from this simple to solve system to the system under consideration, 
namely $\textrm{Re}(F) = \textrm{Im}(F)=0$, and tracks the $n^2$ solution paths defined
by this homotopy.  In the present paper, we use {\tt Bertini} \cite{Bertini-book,Bertini-package} to perform path tracking.
Since {\tt Bertini} relies upon numerical floating-point computations, one heuristically 
obtains numerical approximations of the exact solutions.  


In all of the cases under consideration below, one can easily show that 
the system $\textrm{Re}(F) = \textrm{Im}(F)=0$ has exactly $n^2$ distinct solutions in $\CC^2$.  
That is, each of the $n^2$ solution paths converge in $\CC^2$ to a distinct nonsingular isolated
solution.  Therefore, one can produce an {\em a posteriori} certificate that all $n^2$ solutions
have been found via $\alpha$-theory.  Moreover, since the zeros of $F$ are exactly the solutions 
contained in $\RR^2\subset\CC^2$, the zeros of $F$ can be provably counted using {\tt alphaCertified} \cite{2010arXiv1011.1091H}.




\section{Results}
\subsection{Examples with many zeros}\label{s:extremal}

Counterexamples to Wilmshurst's conjecture have been recently found \cite{LLL} for the $m=n-3$ case.
Here, we generalize this construction of special harmonic polynomials having many zeros
to produce a family of harmonic polynomials $F$, one for each pair $(n,m) = (n,n-\ell)$.  
For finitely many $m$ and $n$, we test these examples using the certified procedure
discussed in Section~\ref{sec:Certified}
and also give an analytic proof of a lower bound on the number of zeros when $\ell$ is fixed and $n$ is large.
In the examples tested, for each $n$, the Wilmshurst conjecture is violated for a broad range of $m$
with the excess most dramatic when $m$ is half of $n$.

The construction of the extremal example proceeds as follows. 
For integers $n>\ell>0$ and number $a$, 
consider the polynomial $f_{n,\ell}$ defined by:
$$f_{n,\ell}(z)=(z-a)^{n-\ell+1}P_{n,\ell}(z) \quad \hbox{where} \quad P_{n,\ell}(z)=\sum_{k=0}^{\ell-1}{\binom{n-\ell+k}{k}}a^kz^{\ell-k-1}.$$

\begin{prop}\label{prop:poly}
With the above choice, the polynomial $f_{n,\ell}$ satisfies:
$$f_{n,\ell}(z)=z^n - a^\ell \binom{n}{\ell} z^{n-\ell} + O(|z|^{n-\ell-1}), \quad |z| \rightarrow \infty .$$
\end{prop}

Let us define:
 \beq \label{eq:special}q(z)=z^n-f_{n,\ell}(z)\quad \textrm{and}\quad p(z)=z^n+f_{n,\ell}(z),\eeq
and consider the corresponding harmonic polynomial $F$ defined as in \eqref{eq:harmonic}.
Note that, by Proposition \ref{prop:poly} (which is proved below in Section \ref{sec:proofs}), the degree of $q$ is $m=n-\ell$.
The $\alpha$-certified count for the number of zeros of $F$
is provided in Figure~\ref{fig:extreme}, where $a$ is chosen to be a small random number.
The excess of solutions with respect to Wilmshurst's prediction is collected in Figure~\ref{fig:excess}. 


\begin{center}
\begin{figure}[!hb]
{\Small
\hspace{-15pt}
\begin{tabular}{c*{20}{|c}}
 $m$ $\backslash$ $n$  & 7  & 8   & 9 & 10 & $11$  & $12$ & $13$ & $14$ & $15$ & $16$ & $17$  & $18$ & $19$ & $20$ \\
\hline
6  & 49 & 52 & 57 & 64 & 69 & 72 & 77 & 88 & 89 & 92 & 97 & 102 & 105 & 112 \\
7  &    & 64	& 67 & 76 & 79 & 88 & 91 & 96 & 103 & 104 & 115 & 124 & 127 & 132\\
8  &    &    & 81 & 84 & 93 & 100 & 105 & 112 & 121 & 124 & 129 & 140 & 141 & 148 \\
9  &    &    &    & 100 & 107 & 112 & 119 & 128 & 135 & 140 & 151 & 156 & 163 & 168 \\
10 &    &    &    &     & 121 & 128 & 133 & 140 & 149 & 160 & 169 & 172 & 185 & 192 \\
11 &    &    &    &     &     & 144 & 151 & 160 & 167 & 180 & 183 & 192 & 207 & 212 \\
12 &    &    &    &     &     &     & 169 & 176 & 185 & 192 & 205 & 216 & 221 & 232 \\
13 &    &    &    &     &     &     &     & 196 & 203 & 212 & 223 & 236 & 243 & 256 \\
14 &    &    &    &     &     &     &     &     & 225 & 232 & 245 & 254 & 265 & 276 \\
15 &    &    &    &     &     &     &     &     &     & 256 & 263 & 276 & 291 & 300 \\
16 &    &    &    &     &     &     &     &     &     &     & 289 & 298 & 309 & 324 \\ 
17 &    &    &    &     &     &     &     &     &     &     &     & 324 & 335 & 348 \\
18 &    &    &    &     &     &     &     &     &     &     &     &     & 361 & 372 \\
19 &    &    &    &     &     &     &     &     &     &     &     &     &     & 400 
\end{tabular}
}
\caption{A table with the $\alpha$-certified count of the number of solutions of $F=0$ for the special choice of $p,q$ given by \eqref{eq:special}.}
\label{fig:extreme}
\end{figure}
\end{center}

\begin{center}
\begin{figure}[!ht]
{\Small
\hspace{-15pt}
\begin{tabular}{c*{20}{|c}}
 $m$ $\backslash$ $n$ & 7  & 8   & 9 & 10 & $11$  & $12$ & $13$ & $14$ & $15$ & $16$ & $17$  & $18$ & $19$ & $20$ \\
\hline
6  & 0 & 0 & 2 & 6 & 8 & 8 & 10 & 18 & 16 & 16 & 18 & 20 & 20 & 24 \\
7  &    & 0	& 0 & 6 & 6 & 12 & 12 & 14 & 18 & 16 & 24 & 30 & 30 & 32\\
8  &    &   & 0& 0 & 6 & 10 & 12 & 16 & 22 & 22 & 24 & 32 &  30 & 34 \\
9  &    &    &    & 0 & 4 & 6 & 10 & 16 & 20 & 22 & 30 & 32 & 36 & 38 \\
10 &    &    &    &     & 0 & 4 & 6 & 10 & 16 & 24 & 30 & 30 & 40 & 44 \\
11 &    &    &    &     &     & 0 & 4 & 10 & 14 & 24 & 24 & 30 & 42 & 44 \\
12 &    &    &    &     &     &     & 0 & 4 & 10 & 14 & 24 & 32 & 34 & 42 \\
13 &    &    &    &     &     &     &     & 0 & 4 & 10 & 18 & 28 & 32 & 42 \\
14 &    &    &    &     &     &     &     &     & 0 & 4 & 14 & 20 & 28 & 36 \\
15 &    &    &    &     &     &     &     &     &     & 0 & 4 & 14 & 26 & 32 \\
16 &    &    &    &     &     &     &     &     &     &     & 0 & 6 & 14 & 26 \\ 
17 &    &    &    &     &     &     &     &     &     &     &     & 0 & 8 & 18 \\
18 &    &    &    &     &     &     &     &     &     &     &     &     & 0 & 8 \\
19 &    &    &    &     &     &     &     &     &     &     &     &     &     & 0
\end{tabular}
}
\caption{A table with the difference between the $\alpha$-certified count (given in Figure~\ref{fig:extreme}) and Wilmshurst's prediction. Thus, this shows the excess of our example to the conjectured bound \eqref{eq:wilm}.}
\label{fig:excess}
\end{figure}
\end{center}

For fixed $\ell$, we prove (see Section \ref{sec:proofs}) the following lower bound for all even $n$ sufficiently large.

 \begin{thm}\label{thm:main}
For $m=n-\ell$ and $\ell$ odd, let $p(z)$ and $q(z)$ be 
the polynomials defined by (\ref{eq:special}).
There exists a number $0<c<1$, such that for all even $n$ sufficiently large, 
the number of solutions $N_F$ to the equation $p(z)-\overline{q(z)} = 0$ satisfies the lower bound:
\begin{equation}\label{maineq}
N_F \geq n^2-2n \ell (1-c).
\end{equation}
\end{thm}

Furthermore, we note that when $\ell=0$ ($m=n-1$) the excess is zero corresponding to the zeros on the diagonal
of Figure~\ref{fig:excess}.  This is due to Wilmshurst's bound being correct in this case 
and our polynomials reduce to examples that were used to show that the bound is sharp \cite{Wilm98}.
Moreover, Figure~\ref{fig:excess} also reveals that the intermediate values have 
especially high excess when $m$ is close to be a half of $n$.
If we restrict to the case $n=2m$, we notice that the numbers appearing in Figure~\ref{fig:extreme} 
are generated by a simple formula.  For $n=12,14,16,18,20$, the entries are $72,96,124,156,192$
which  are all equal to $n^2/2 - n + 12$.  We conjecture that this holds for all $n$.

\smallskip

\noindent {\bf Conjecture:} For all even $n=2m=2\ell$,
the polynomial $F$ defined as above using $f_{n,\ell}$
has exactly $N_F =n^2/2 - n + 12$ many zeros.

\smallskip

Suppose that some modified form of the Wilmshurst conjecture is true,
and there exists an improvement on the B\'ezout bound
that is a quadratic polynomial in $m$ and $n$ which is linear in $n$ for each fixed $m$.
Then, the above conjecture (which is confirmed by the data up to $n=20$)
indicates that there is a cross term $n \cdot m$ in this bound,
since for $n=2m$, we have $n^2/2 = n\cdot m$.
The presence of a cross term was previously suggested by the upper bound conjectured in \cite{LLL}
and restated here (note that this statement only improves the B\'ezout bound for $m$ in the range $m < n/2$).

\smallskip

\noindent {\bf Conjecture: (\cite[Introduction]{LLL})} Let $F$ be a general harmonic polynomial of the form \eqref{eq:harmonic}
with $n>m$.  Then, the number $N_F$ of zeros of $F$ satisfies 
$$N_F \leq 2m(n-1) + n.$$

\subsection{Random harmonic polynomials}\label{s:random}

The following is prompted by the work of Li and Wei \cite{LiWei2009} mentioned above.
The Li-Wei model is deficient in its average number of zeros when $m = \alpha n$,
and it is desirable to find an ensemble of random harmonic polynomials $F$ with more zeros on average.  Figure~\ref{fig:random} shows
the mean and standard deviation upon performing $1000$ trials for each 
stated $m$ and $n$ using the Li-Wei model.
An intuitive reason for the outcome $\EE N_F \sim n$ when $m= \alpha n$ in this 
model is that in choosing two different sets of binomial coefficients, the orientation-reversing term $\ol{q(z)}$ is asymptotically negligible.\footnote{When $m=\alpha n$, 
the numbers $\binom{n}{k}$ become much larger than $\binom{m}{k}$.}
The most interesting aspect of the Wilmshurst problem is the interaction between $p$ and $q$ leading to widespread changes in orientation of the harmonic mapping $F:\CC \rightarrow \CC$.
Thus, it seems natural to try to modify $q$ in such a way that it is more comparable to $p$.
We pose this as a problem.

 \begin{center}
\begin{figure}[!h]
{\Small
\hspace{-15pt}
\begin{tabular}{c*{20}{c}}
 $m$ $\backslash$ $n$   & 10  & 15 & 20 & 25  & 30 \\
\hline
5  & (10.31,0.77) & (15.10,0.45) & (20.05,0.30) & (25.03,0.25) & (30.01,0.11) \\
6  & (10.64,1.06) & (15.13,0.50) & (20.09,0.41) & (25.04,0.29) & (30.03,0.24) \\
7  & (11.24,1.48) & (15.25,0.69) & (20.11,0.47) & (25.04,0.29) & (30.03,0.26) \\
8  & (12.97,2.21) & (15.40,0.83) & (20.16,0.57) & (25.07,0.36) & (30.05,0.31) \\
9  & (17.70,3.69) & (15.67,1.12) & (20.22,0.66) & (25.13,0.49) & (30.07,0.38) \\
10 & (25.91,4.95) & (16.00,1.35) & (20.30,0.77) & (25.13,0.52) & (30.09,0.42) \\
11 &  & (16.87,1.77) & (20.49,0.95) & (25.21,0.63) & (30.12,0.49) \\
12 &  & (18.25,2.36) & (20.60,1.06) & (25.32,0.74) & (30.16,0.55) \\
13 &  & (22.27,3.53) & (20.86,1.26) & (25.36,0.81) & (30.17,0.58) \\
14 &  & (31.19,5.31) & (21.53,1.58) & (25.50,0.94) & (30.27,0.71) \\
15 &  & (46.79,7.61) & (22.16,1.97) & (25.62,1.10) & (30.32,0.75) \\
16 &  &  & (23.65,2.53) & (25.91,1.25) & (30.38,0.82) \\
17 &  &  & (26.22,3.10) & (26.28,1.46) & (30.52,0.96) \\
18 &  &  & (32.57,4.70) & (26.82,1.75) & (30.67,1.17) \\
19 &  &  & (47.19,7.46) & (27.57,2.09) & (30.88,1.29) \\
20 &  &  & (71.75,10.61) & (28.76,2.46) & (31.21,1.47) \\
21 &  &  &  & (31.05,3.18) & (31.55,1.63) \\
22 &  &  &  & (34.97,3.96) & (32.00,1.84) \\
23 &  &  &  & (44.00,5.88) & (32.72,2.27) \\
24 &  &  &  & (65.04,9.47) & (33.95,2.47) \\
25 &  &  &  & (100.02,12.82) & (35.81,3.17) \\
26 &  &  &  &  & (38.70,3.72) \\
27 &  &  &  &  & (44.54,4.94) \\
28 &  &  &  &  & (56.87,7.19) \\
29 &  &  &  &  & (85.33,11.26) \\
30 &  &  &  &  & (130.75,15.59) \\
\end{tabular}
}
\caption{Outcomes for five values of $n$, and $5 \leq m \leq n$. Entries in the table are listed as (mean,standard deviation). }
\label{fig:random}
\end{figure}
\end{center}

\noindent {\bf Problem:} Modify $q$ in the Li-Wei definition of \emph{random} such that, when $m=\alpha n$,
$$\EE N_F = \Theta(m^{3/2})=\Theta(n^{3/2}).$$

\smallskip

\medskip

In order to give more weight to $q$, we modified (\ref{eq:LW}) by choosing $q$ to be a truncated i.i.d. copy of $p$:
\begin{equation} \label{eq:LWmodified}
p(z) = \sum_{k=0}^{n} a_k \sqrt{\binom{n}{k}} z^k, \quad q(z) = \sum_{k=0}^{m} b_k \sqrt{\binom{n}{k}} z^k.
\end{equation}

Using this model, 
the mean and standard deviation (performing $1000$ trials for each stated $m$ and $n$) are listed in Figure~\ref{fig:random2}.
When $m=\alpha n$, the truncated model has a much higher mean number of zeros and perhaps even achieves the $\Theta(m^{3/2})$ order of growth.

 \begin{center}
\begin{figure}[!h]
{\Small
\hspace{-15pt}
\begin{tabular}{c*{20}{c}}
  $m$ $\backslash$ $n$     & 10  & 15 & 20 & 25  & 30 \\
\hline
5  & (13.91,2.43) & (18.42,2.17) & (23.23,2.19) & (28.22,2.19) & (33.17,2.20) \\
6  & (15.83,2.95) & (19.75,2.55) & (24.50,2.55) & (29.45,2.48) & (34.31,2.53) \\
7  & (17.95,3.52) & (21.27,2.97) & (25.99,2.94) & (30.72,2.79) & (35.63,2.71) \\
8  & (20.70,4.02) & (23.53,3.46) & (27.63,3.32) & (32.35,3.23) & (36.95,3.18) \\
9  & (25.34,5.26) & (25.66,3.90) & (29.44,3.62) & (33.89,3.38) & (38.48,3.40) \\
10 & (25.98,5.02) & (28.09,4.39) & (31.58,3.95) & (35.86,3.85) & (40.33,3.84) \\
11 &  & (31.25,4.91) & (33.78,4.33) & (37.89,4.05) & (42.04,4.15) \\
12 &  & (35.09,5.45) & (36.70,4.76) & (40.09,4.61) & (44.27,4.48) \\
13 &  & (39.38,6.44) & (39.37,5.12) & (42.47,4.91) & (46.57,4.90) \\
14 &  & (46.20,8.33) & (42.67,5.71) & (44.93,5.13) & (48.84,5.31) \\
15 &  & (46.65,7.27) & (46.22,6.45) & (47.96,5.72) & (51.46,5.42) \\
16 &  &  & (50.75,7.02) & (51.27,6.01) & (54.10,5.83) \\
17 &  &  & (56.19,7.79) & (55.14,6.76) & (57.11,6.18) \\
18 &  &  & (61.98,8.52) & (59.00,7.07) & (60.69,6.56) \\
19 &  &  & (70.71,10.98) & (62.80,7.54) & (63.68,6.99) \\
20 &  &  & (71.76,10.32) & (67.97,7.97) & (67.57,7.58) \\
21 &  &  &  & (73.03,8.64) & (71.38,7.79) \\
22 &  &  &  & (80.27,9.93) & (75.87,8.17) \\
23 &  &  &  & (88.46,11.46) & (80.96,9.11) \\
24 &  &  &  & (98.25,14.13) & (86.50,9.75) \\
25 &  &  &  & (100.02,12.82) & (92.40,10.10) \\
26 &  &  &  &  & (98.74,11.53) \\
27 &  &  &  &  & (106.48,11.94) \\
28 &  &  &  &  & (116.66,14.49) \\
29 &  &  &  &  & (129.80,18.18) \\
30 &  &  &  &  & (130.99,15.76) \\
\end{tabular}
}
\caption{With $q$ taken to be a truncated Kostlan, when $m= \alpha n$, the mean of $N_F$ grows more quickly than in the Li-Wei model.}
\label{fig:random2}
\end{figure}
\end{center}

When $m=n$, (\ref{eq:LWmodified}) coincides with the Li-Wei definition.
The analytic methods used in \cite{LiWei2009}
provided asymptotics for the mean but not the variance.
The experimental data for the variance leads us to the following conjecture 
(see Figure~\ref{fig:var}).

\begin{figure}[!h]
    \begin{center}
    \includegraphics[scale=.6]{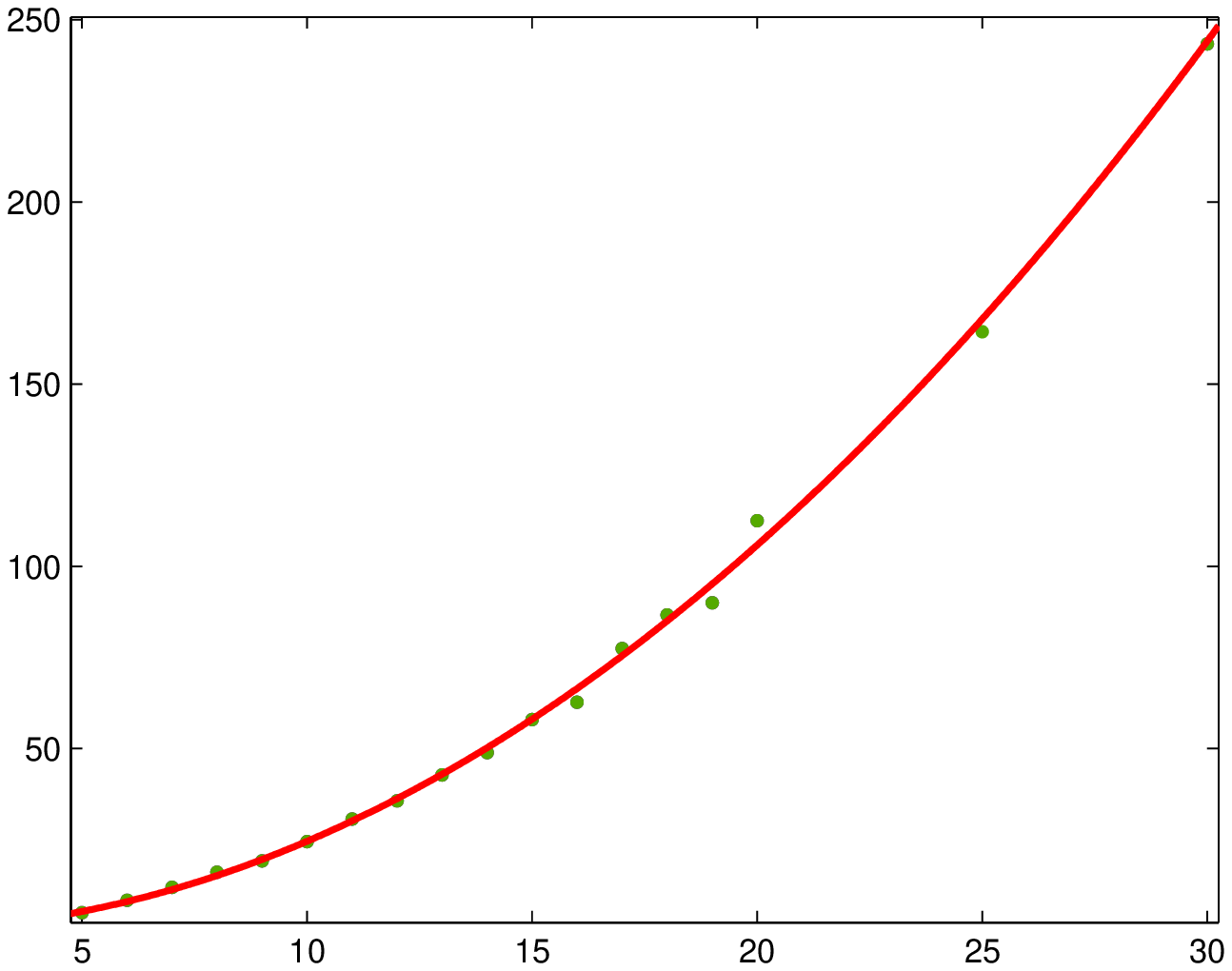}
    \end{center}
    \caption{For $m=n$, the sample variance of the number of zeros and a (least squares fit) quadratic curve. We conjecture the variance to be $\Theta(n^2)$.}
    \label{fig:var}
\end{figure}

\medskip

\noindent {\bf Conjecture:} For $m=n$ the variance of the number of zeros grows quadratically.

\smallskip
\medskip

Since the mean for $m=n$ is $ \EE N_F \sim \frac{\pi}{4} n^{3/2}$,
this conjecture implies that the variance has a higher order of growth than the mean.
This entails a departure from behavior observed in 
studies of random univariate polynomials where the variance has been found to be asymptotically proportional to the mean (see \cite{GW} and the references therein).

\section{Proofs}\label{sec:proofs}

\subsection{Proof of Proposition \ref{prop:poly}}

Fix $n$.  We prove the statement by induction on~$\ell$,
starting from $\ell = n$ as the ``base case'' and working backwards.
Thus, fix $L$ and suppose that the statement in the proposition (induction hypothesis) is true for all $\ell = L+1,L+2,\dots,n$.
Then we will show that the statement is true for $\ell = L$.

In the case $\ell = n$, the binomial coefficients in $P_{n,n}(z)$ are all $\binom{k}{k}=1$,
and it is elementary to check that $f_{n,n}(z) = z^n - a^n$.

\smallskip

\noindent {\bf Claim:} The polynomials $P_{n,\ell}(z)$ satisfy the recurrence relation:
\begin{equation}\label{eq:claim}
(z-a)P_{n,\ell}(z) = P_{n,\ell+1}(z)-a^\ell \binom{n}{\ell} .
\end{equation}

Applying the Claim (in the second line of the next equations),
\begin{align*}
f_{n,L}(z) &= (z-a)^{n-(L+1)-1}(z-a) P_{n,L}(z) \\
     &= (z-a)^{n-(L+1)-1} \left( P_{n,L+1}(z) - a^L \binom{n}{L} \right) \\
     &= (z-a)^{n-(L+1)-1} P_{n,L+1}(z) - a^L \binom{n}{L} (z-a)^{n-L}.
\end{align*}
Applying the induction hypothesis (which was assumed true for $\ell = L+1$):
\begin{align*}
f_{n,L}(z) &= z^n-a^{L+1}\binom{n}{L+1} z^{n-(L+1)} + O(z^{n-(L+2)}) - a^L \binom{n}{L} (z-a)^{n-L} \\
     &= z^n - a^L \binom{n}{L} (z-a)^{n-L} + O(z^{n-(L+1)}).
\end{align*}
This completes the inductive step.  It remains only to prove the Claim.

We apply the elementary identity (connected with Pascal's triangle):
$$ \binom{n-\ell+k}{k} = \binom{n-(\ell+1)+k}{k} + \binom{n-\ell + k - 1}{k-1} ,$$
where the second term on the right hand side is taken to be zero when $k=0$.
Applying this to $z P_{n,\ell}(z)$:
\begin{align*}
z P_{n,\ell}(z) &= z \sum_{k=0}^{\ell-1} \binom{n-\ell + k}{k} a^k z^{\ell-k-1} \\
&= z \sum_{k=0}^{\ell-1} \binom{n-(\ell+1)+k}{k} a^k z^{\ell-k-1} +z \sum_{k=1}^{\ell-1} \binom{n-\ell + k-1}{k-1} a^k z^{\ell-k-1} .\\
\end{align*}
Distributing the factor of $z$ in both sums on the right hand side, and taking one factor of $a$ outside the second sum leads to:
\begin{equation*}
zP_{n,\ell}(z) = \sum_{k=0}^{\ell-1} \binom{n-(\ell+1)+k}{k} a^k z^{\ell-k} + a \sum_{k=1}^{\ell-1} \binom{n-\ell + k-1}{k-1} a^k z^{\ell-k} .\\
\end{equation*}
Shifting the index in the second sum:
\begin{align*}
zP_{n,\ell}(z) &= \sum_{k=0}^{\ell-1} \binom{n-(\ell+1)+k}{k} a^k z^{\ell+1-k-1} + a \sum_{k=0}^{\ell-2} \binom{n-\ell + k-1}{k-1} a^k z^{\ell-k-1}\\
&= P_{n,\ell+1} - a^{\ell} \binom{n-1}{\ell}+ a \left( P_{n,\ell}- a^{\ell-1}\binom{n-1}{\ell-1}\right) \\
&= P_{n,\ell+1} - a^{\ell} \binom{n}{\ell} + a P_{n,\ell} ,
\end{align*}
where we have used the Pascal's triangle identity again in the last line.
This proves the Claim.

\subsection{Proof of Theorem \ref{thm:main}}

The proof is an extension of \cite{LLL},
and we follow the same notation and terminology used there.
We focus on the zero set of the imaginary part of $f_{n,\ell}(z)$,
since zeros of $F$ are also intersections between the sets where $\Rea z^n = 0$ and $\Ima f_{n,\ell}(z)=0$.
For the moment, we assume $a$ is purely real.
In what follows, we will use $i := \sqrt{-1}$ to denote the imaginary unit.

Following \cite{LLL},
we first note that the zero set 
$$\{z: \Ima f_{n,\ell}(z) = 0 \}$$ 
consists of $n$ smooth curves: each joins two of the points at infinity in the directions 
$e^{\ii\pi j /n}$ and $e^{\ii\pi k /n}$ for some $j,k\in\{0,1,\cdots,2n-1\}$ such that $j\neq k$.
This follows from the asymptotic behavior of $f_{n,\ell}(z) \sim z^n$, $z \rightarrow \infty$.
Let us use the notation $\infty \times e^{\ii\theta}$ for ``the limit taken in the direction $e^{\ii\theta}$.''
For any $j\in{\mathbb Z}$, $\infty \times e^{\ii\pi j /n}$ is visited by a single such curve.
Let us call each of these $n$ curves a \emph{line}.
We will also use the term \emph{ray} to refer to one of the two curves formed by removing a point from such a line.

For each $n$, there are $n-\ell$ distinct lines that intersect at $a\in\CC$.
Indeed, $a$ is a zero of $f_{n,\ell}$ while also being a critical point with multiplicity $n-\ell$,
which implies that the zero set near $a$ is an intersection of $n-\ell$ smooth curves,
no two of which could be different segments of the same line.
Otherwise, the zero set of $\{z: \Ima f_{n,\ell}(z) = 0 \}$ would enclose a bounded region which violates the maximum principle for harmonic functions.
Note that $f_{n,\ell}$ has $\ell-1$ other critical points besides $a$.
The locations of these critical points are the only other possible intersections among the lines.

The remaining $\ell$ lines that do not pass through $a$ must each pass through some zero of $P_{n,\ell}$,
since the real part of $f_{n,\ell}$ changes sign between any two $\infty \times e^{\ii\pi j /n}$ and $\infty \times e^{\ii\pi k /n}$ with $j\neq k$.


\begin{lemma}\label{lemma:c}
There exists a constant $c$ such that, for each $k=n,n\pm1,\dots,n\pm \lceil cn \rceil $,
the line starting from $\infty \times e^{\ii\pi k /n}$ passes through $a$.
\end{lemma}

We first show how the lemma is applied before proving it.

Note that for $\{k=0,\pm 1, \pm 2,..,\pm (n-1),n\}$,
a ray starting from $a$ and diverging to $\infty \times e^{\ii\pi k /n}$ intersects $\Rea z^n=0$ at least $|k|$ times (cf. \cite[Lemma 5]{LLL}).
For $k=n$ this statement requires slightly perturbing the imaginary part of $a$ (otherwise the corresponding ray intersects $\Rea z^n$ at the origin).
We can thus estimate the total number of zeros of $F$ by summing these values and subtracting the $2\ell$ terms corresponding to the $\ell$ lines that do not pass through $a$.
Lemma \ref{lemma:c} shows that if the line in the direction $\infty \times e^{\ii\pi k /n}$ does not pass through $a$, 
then $|k| \leq (n-cn)$.
From this it follows that the $2 \ell$ exceptional terms sum to at most $2 \ell (n-cn)$.
We thus have the following:
\begin{align*}
N_F \geq n + 2 \sum_{k=1}^{n-1} k - 2 \ell (n-cn) = n^2 - 2n \ell ( 1-c).
\end{align*}

\begin{proof}[Proof of Lemma \ref{lemma:c}]

We first establish the following.

\smallskip

\noindent {\bf Claim:} There exists a sector $S$ centered at the point $a$ that opens to the left with top and bottom edges having angles $-\pi \pm \varepsilon$ (so $S$ contains the negative real axis),
and $S$ does not contain any zeros of $P_{n,\ell}$ or critical points of $f_{n,\ell}$.
Moreover, the choice of $\varepsilon$ can be made independent of $n$.

\smallskip

 The existence of $S$ follows from the fact that, for $\ell$ fixed and $n \rightarrow \infty$,
 the zeros of $P_{n,\ell}(a nz)$ and the critical points of $f_{n,\ell}(a nz)$ different from $1/n$ each converge to points (independent of $n$) none of which are real.
 
 In order to see this we look at the polynomials $\frac{(n-\ell)!(\ell-1)!}{(n-1)!}P_{n,\ell}(n a z)$  (which have the same zeroes as $P_{n,\ell}(a nz)$). The coefficients of these polynomials converge to:
  $$\lim_{n\to \infty}\frac{{n-k-1\choose \ell-k-1}}{{n-1\choose \ell-1}}a^{\ell-1}n^k= \frac{(\ell-1)!}{(\ell-k-1)!}a^{\ell-1}.$$
 In particular the zeroes of $P_{n,\ell}( anz)$ converge to the zeros of the polynomial:
 $${Q}_\ell(z)=\sum_{k=0}^{\ell-1} \frac{(l-1)!}{(\ell-k-1)!}z^k=e^{1/z}z^{\ell-1}\Gamma(\ell, 1/z)$$
 where in the last identity we have used the definition of \emph{incomplete Gamma function} \cite{AS}:
 $$\Gamma(\ell, 1/x)=\int_{1/x}^\infty e^{-t}t^{l-1}dt,\qquad x\in \mathbb{R}.$$
Notice that $Q_\ell(0)=1$; moreover since we have chosen $\ell$ to be odd, then the integrand in the above expression is always positive. Thus $Q_\ell$ has no \emph{real} zeros. This proves the claim for the location of the zeros of $P_{n,\ell}(anz)$.
 For the critical points of $f_{n,\ell}(z)$ different from $a$ we notice that these are defined by (see \cite{LLL} for the details of this construction in the case $\ell=3$):
 $$\frac{f_{n,\ell}'(z)}{(z-a)^{n-\ell}}=0.$$
 Evaluating the polynomial on  left hand side of the above equation at $anz$ and dividing by ${n-1\choose \ell-1}$, a simple computation shows that the coefficients again converge to the coefficients of $Q_\ell$: thus the location of these critical points is asymptotically fixed and none of them is real.

Applying the Claim, there is some constant $C$ determined by the angles $\pi \pm \varepsilon$ 
such that all directions $\infty \times e^{\ii\pi k /n}$ with $k=n,n\pm1,n\pm2,..,n\pm \lfloor Cn \rfloor$,
are between the edges of $S$.
We will show that a positive fraction (independent of $n$) of these lines pass through the point $a$.
Suppose that one of the lines  (and therefore also the line symmetric wrt the real axis) does not pass through $a$,
and let $j$ denote the smallest positive integer such that the lines $L_{\pm}$ 
starting from the directions $\infty \times e^{\ii\pi (n \pm j) /n}$
do not pass through $a$.
Since $L_\pm$ each pass through a zero of $P_{n,\ell}$, 
the Claim implies that $L_\pm$ exit the sector $S$.
Let $z_\pm$ denote the two (symmetric wrt the real axis) points where this occurs.
For $k= n\pm (j+1),..,n\pm \lfloor Cn \rfloor$, any line starting from $\infty \times e^{\ii\pi k /n}$ must exit $S$.
Otherwise, any line that does not exit would intersect $L_\pm$, but there are no critical points in $S$ by the Claim.
The lines from directions corresponding to $k= n, n\pm1,..n\pm (j-1)$
are thus the only lines that pass between $z_\pm$.
On the other hand, the number of lines passing between $z_\pm$
can be measured more directly as the increment of the argument of $f_{n,\ell}$ along the vertical segment joining the points $z_\pm$ (cf. \cite{LLL}).
We have:
\begin{align*}
	\Delta \arg f(z) &= (n-\ell+1) \left[ \arg (z_+ - a) - \arg(z_- - a) \right]  + \Delta \arg P_{n,\ell}(z) \\
	&\geq n C' - 4 \pi \ell,
\end{align*}
Since $j$ was assumed to be the smallest exceptional case, this shows that there is a positive fraction of the lines pass through $a$,
which establishes the lemma.

\end{proof}

\end{document}